\documentclass[12pt]{amsart}

\usepackage{amssymb}
\usepackage{mathrsfs}
\usepackage{hyperref}

\makeatletter
\def\newrefformat#1#2{%
  \@namedef{pr@#1}##1{#2}}
\def\fref#1{\@prettyref#1:}
\def\@prettyref#1:#2:{%
  \expandafter\ifx\csname pr@#1\endcsname\relax%
    \PackageWarning{prettyref}{Reference format #1\space undefined}%
    \ref{#1:#2}%
  \else%
    \csname pr@#1\endcsname{#1:#2}%
  \fi%
}
\makeatother

\newrefformat{eq}{\textup{(\ref{#1})}}
\newrefformat{ax}{\textup{\ref{#1}}}
\newrefformat{chap}{Chapter~\ref{#1}}
\newrefformat{sec}{Section~\ref{#1}}
\newrefformat{apdx}{Appendix~\ref{#1}}
\newrefformat{tab}{Table~\ref{#1} on page~\pageref{#1}}
\newrefformat{fig}{Figure~\ref{#1} on page~\pageref{#1}}

\newcommand{\mynewthm}[3][dummythm]{%
  \newtheorem{#2}[#1]{#3}%
  \newrefformat{#2}{#3~\ref{##1}}%
}

\theoremstyle{plain}
\mynewthm{thm}{Theorem}
\mynewthm{prp}{Proposition}
\mynewthm{lem}{Lemma}
\mynewthm{fct}{Fact}
\mynewthm{cor}{Corollary}

\theoremstyle{definition}
\mynewthm{dfn}{Definition}
\mynewthm{conv}{Convention}
\mynewthm{ntn}{Notation}
\mynewthm{conj}{Conjecture}
\mynewthm{cst}{Construction}

\theoremstyle{remark}
\mynewthm{rmk}{Remark}
\mynewthm{qst}{Question}
\mynewthm{exm}{Example}

\newcommand{\myenumlabel}[1]{\textnormal{(\roman{#1})}}

\newcounter{cycprfcnt}
\newenvironment{cycprf}%
{\begin{list}{\PackageWarning{pezz}{Label required for cycprf}}%
  {%
    \setcounter{cycprfcnt}{1}
    \setlength{\itemindent}{0.5\leftmargin}%
    \setlength{\leftmargin}{0pt}%
    \newcommand{\cpcurr}{\myenumlabel{cycprfcnt}}%
    \newcommand{\cpnext}{\addtocounter{cycprfcnt}{1}\cpcurr}%
    \newcommand{\impnext}{\cpcurr{} $\Longrightarrow$ \cpnext.}%
  }%
}%
{\qedhere\end{list}}%

\makeatletter

\def\indsym#1#2{%
  \setbox0=\hbox{$\m@th#1x$}%
  \kern\wd0%
  \hbox to 0pt{\hss$\m@th#1\mid$\hbox to 0pt{$\m@th#1^{#2}$}\hss}%
  \lower.9\ht0\hbox to 0pt{\hss$\m@th#1\smile$\hss}%
  \kern\wd0}
\newcommand{\ind}[1][]{\mathop{\mathpalette\indsym{#1}}}

\def\nindsym#1#2{%
  \setbox0=\hbox{$\m@th#1x$}%
  \kern\wd0%
  \hbox to 0pt{\hss$\m@th#1\not$\kern1.4\wd0\hss}
  \hbox to 0pt{\hss$\m@th#1\mid$\hbox to 0pt{$\m@th#1^{#2}$}\hss}%
  \lower.9\ht0\hbox to 0pt{\hss$\m@th#1\smile$\hss}%
  \kern\wd0}
\newcommand{\nind}[1][]{\mathop{\mathpalette\nindsym{#1}}}

\def\dotminussym#1#2{%
  \setbox0=\hbox{$\m@th#1-$}%
  \kern.5\wd0%
  \hbox to 0pt{\hss\hbox{$\m@th#1-$}\hss}%
  \raise.6\ht0\hbox to 0pt{\hss$\m@th#1.$\hss}%
  \kern.5\wd0}

\makeatother

\DeclareMathOperator{\tp}{tp}

\DeclareMathOperator{\tS}{S}
\DeclareMathOperator{\dcl}{dcl}

\DeclareMathOperator{\SU}{SU}

\renewcommand{\setminus}{\smallsetminus}
\def\models{\vDash}

\newcommand{\rest}{{\restriction}}
\newcommand{\Cb}{\mathrm{Cb}}

\newcommand{\fP}{\mathfrak{P}}

\newcommand{\setR}{\mathbb{R}}

\newcommand{\Dx}[2][\Xi]{D(#2,#1)}
\DeclareMathOperator{\mcl}{mcl}

\title{On supersimplicity and lovely pairs of cats}

\author{Itay Ben-Yaacov}

\address{Itay Ben-Yaacov\\
  University of Wisconsin -- Madison\\
  Department of Mathematics\\
  480 Lincoln Drive\\
  Madison, WI 53706\\
  USA}
\urladdr{http://www.math.wisc.edu/\textasciitilde pezz}

\date{\today}

\begin{document}

\begin{abstract}
  We prove that the definition of supersimplicity in metric structures
  from \cite{pezz:Morley} is equivalent to an \textit{a priori}
  stronger variant.
  This stronger variant is then used to prove that if $T$ is a
  supersimple Hausdorff cat then so is its theory of lovely pairs.
\end{abstract}

\maketitle

\section*{Introduction}

A superstable first order theory is one which is stable in every large
enough cardinality, or equivalently, one which is stable (in some
cardinality), and in which the type of every finite tuple over
arbitrary sets does not divide over a finite subset.
In more modern terms we would say that a first order theory is
superstable if and only if it is stable and supersimple.

Stability and simplicity were extended to various non-first-order
settings by various people.
Stability in the setting of large homogeneous structures goes back a
long time (see \cite{sh:lazy}), and some aspects of simplicity theory
were also shown to hold in this setting in \cite{bl:hom}.
The setting of compact abstract theories, or cats, was introduced in
\cite{pezz:posmod} with the intention, among others, to provide a
better non-first-order setting for the development of simplicity
theory, which was done in \cite{pezz:catsim}, and under the additional
assumption of thickness (with better results) in
\cite{pezz:fnctcat}.

\emph{Hausdorff} cats are ones whose type spaces are Hausdorff.
Many classes of metric structures arising in analysis can be viewed as
Hausdorff cats (e.g., the class of probability measure algebras
\cite{pezz:schroed}, elementary classes of Banach space
structures in the sense of Henson's logic, etc.)
Conversely,
a Hausdorff cat in a countable language
admits a definable metric on its home sort
which is unique up to uniform equivalence of
metrics \cite{pezz:Morley}
(and even if the language is uncountable this result remains
essentially true).
Thus Hausdorff cats form a natural setting for the study of metric
structures.

There is little doubt about the definitions of stability and
simplicity in the case of (metric) Hausdorff cats: all the approaches
mentioned above, and others, agree and give essentially the same
theory as in first order logic.
Many natural examples are indeed stable.
Unfortunately, no metric structure can be superstable or supersimple
according to the classical definition, unless it is essentially
discrete:
Indeed, in most cases that $b_n \to a$ in the metric we have
$a \nind_{b_{<n}} b_{<\omega}$ for all $n$.

A very illustrative example is the following.
Let $T$ be a first order theory.
For every $M \models T$, we can view $M^\omega$ as a metric structure, with
$$d(a_{<\omega},b_{<\omega}) = \inf \{2^{-n}\colon a_{<n} = b_{<n}\}.$$
The class of metric structure $\{M^\omega\colon M \models T\}$ is the class of
complete models (in the sense of \cite{pezz:Morley}) of a compact
abstract theory naturally called $T^\omega$.
If $T$ is stable (simple), then so is $T^\omega$, but $T^\omega$ is never
supersimple: this can be seen using the argument in the preceding
paragraph, or directly from the fact that finite tuples in the sense
of $T^\omega$ are in fact infinite tuples in the sense of $T$ (these two
arguments eventually boil down to the same thing).

In an arbitrary metric cat
define that $a^\varepsilon \ind_C B$ if there is $a'$ such
that $d(a,a') \leq \varepsilon$ and $a' \ind_C B$ (later on we will slightly
modify this).
Then for every simple first order theory $T$ and the corresponding
$T^\omega$ we have:
$$a_{<\omega}^{(2^{-n})} \ind_C B \text{ (in the sense of $T^\omega$)}
\Longleftrightarrow a_{<n} \ind_C B \text{ (in the sense of $T$)}.$$
(Knowing $a_{<\omega} \in M^\omega$ up to
distance $2^{-n}$ is the same as knowing $a_{<n}$.)
It follows that $T$ is supersimple if and only if, in $T^\omega$,
for every $\bar a \in M^\omega$,
$\varepsilon > 0$, and set $B$, there is $B_0 \subseteq B$ finite
such that $\bar a^\varepsilon \ind_{B_0} B$.

Generalising from this example, we suggested in \cite{pezz:Morley}
that:\\
-- Finite tuples in metric structures behave in some sense like
infinite tuples in classical first
order structures, and the right way to extract a ``truly finite'' part
of them is to consider them only up to some positive distance.\\
-- As a consequence, the above characterisation of the supersimplicity
of $T$ by properties of $T^\omega$ should be taken as the
\emph{definition} of the supersimplicity of a metric theory (so $T^\omega$
would be supersimple if and only if $T$ is).

One can now define that a Hausdorff cat is \emph{superstable} if it is
stable and supersimple.
An alternative approach to superstability was suggested by Iovino
in the case of Banach space structures through the re-definition of
$\lambda$-stability in a manner that takes the metric into account
\cite{iov:stab}.
The two definitions agree: $T$ is $\lambda$-stable for all big enough $\lambda$
(by Iovino) if and only if it is stable and supersimple by the
definition above.
(This follows from the metric stability spectrum theorem
\cite[Theorem~4.13]{pezz:Morley} as in the classical case.)
We find this a fairly reassuring evidence that the definitions are
indeed ``correct''.

\medskip\noindent
The present paper attempts to address some questions these definitions
raise:

First, the definition of supersimplicity above is somewhat disturbing,
as it translates the two occurrences of ``finite'' in the original
definition differently.
We would prefer something of the form: ``$T$ is supersimple if
for every $a$, $\varepsilon > 0$ and $B$, there is a distance $\delta > 0$
such that $a^\varepsilon \ind_{B^\delta} B$.''
Of course, in order to do that we would first have to give meaning to
$a^\varepsilon \ind_{B^\delta} B$.
This is addressed in \fref{sec:SS}.

A second issue arises from the theory $T^P$ of beautiful (or lovely)
pairs of models of a stable (or simple) theory $T$
\cite{poiz:paires,ppv:pairs,pezz:nfopairs}.
It was shown (by Buechler \cite{bue:PsProj}, later extended by
Vassiliev \cite{vas:su1p}) that such theories of pairs of models of
$T$ can be used as means for obtaining information on $T$ itself.
More precisely, for a rank one theory $T$, the rank of $T^P$ yields
information about the geometry of $T$.
While metric structures can never have ``rank one'', it is natural to
seek to compare the rank of $T^P$ with that of $T$, when $T$ is
superstable or supersimple.
In order for such a course of action to be feasible, one would first
have to show that in that case $T^P$ is supersimple
as well.

In \cite{ppv:pairs} and in \cite{pezz:nfopairs} two distinct proofs
are given to the effect that if $T$ is a supersimple first order
theory, or more generally, a ``supersimple'' cat in the wrong sense
that does not take into account the metric, then so is the theory of
its lovely pairs $T^P$.
Due to the nature of independence in $T^P$, both proofs inevitably use
the fact that $T$ is supersimple at least twice.
These proofs do not extend to the corrected definition of
supersimplicity: without entering into details, on the first
application of supersimplicity we see that $a^\varepsilon \ind_{B_0} B$ for some
finite $B_0 \subseteq B$, but then we cannot apply supersimplicity to
the type of $B_0$ over something else.
On the other hand, if we did have $a^\varepsilon \ind_{B^\delta} B$ for some
$\delta > 0$, we could apply supersimplicity to types of $B^\delta$ over
another set, and the proof may be salvaged.
This is addressed in \fref{sec:SSPairs}.

Thus the novelty of this paper is a notion of
independence over virtual tuples, i.e., over objects of the form
$B^\delta$.
This is a venture into difficult and unsound terrain (for example,
the results of \cite{pezz:ultind}, while dealing
with ultraimaginary elements rather than virtual ones,
suggest that independence over objects which are not ``at least''
hyperimaginary should be approached with extreme caution and without
too many hopes).
While one can come up with many definitions for such a notion of
independence it is not at all obvious to
come up with one which satisfies the usual axioms, or even any
``large'' subset thereof.
Our notion of independence is merely shown to satisfy some
partial transitivity properties (\fref{prp:eIndTrans}), and at the
same time to yield an equivalent characterisation of supersimplicity
(\fref{thm:SSChar}).
We content ourselves with such a modest achievement as it does suffice
to close the gap in the proof that if $T$ is supersimple then so is
$T^P$ (\fref{thm:TPSS}).

Other properties, such as symmetry, are lost (when considering
independence over virtual tuples).
For example, in a
Hilbert space, if $u$ and $v$ are unit vectors and $\varepsilon \leq \sqrt{2}$,
then one can show that:
\begin{align*}
  u \ind_{v^\varepsilon} v & \Longleftrightarrow \|u-v\| \geq \varepsilon \\
  v \ind_{v^\varepsilon} u & \Longleftrightarrow \|u-v\| = \sqrt{2} \; (\Longleftrightarrow u\ind v).
\end{align*}
The question of finding a notion of independence which has more of the
usual properties (e.g., symmetry, full transitivity, extension)
without losing those we need for the results presented in this paper,
remains an open (and difficult) one.

\medskip\noindent

We assume basic familiarity with the setting of compact abstract
theory and simplicity theory in this setting (see \cite{pezz:bslcat}).

We use $a,b,c,\ldots$ to denote possibly infinite tuples of elements in
the universal domain of the theory under consideration.
When we want them to stand for a single element, we say so explicitly.
Similarly, $x,y,z,\ldots$
denote possibly infinite tuples of variables.
Greek letters $\varepsilon,\delta,\ldots$ denote values in the interval $[0,\infty]$, or
possibly infinite tuples thereof.

\section{Supersimplicity}
\label{sec:SS}

\begin{conv}
  We work in a Hausdorff cat $T$.
\end{conv}

We recall from \cite{pezz:Morley} that every sort admits a definable
metric (i.e., a metric $d$ such that for every $r \in \setR^+$, the
properties $d(x,y) \leq r$ and $d(x,y) \geq r$ are type-definable), or, if
not, can be decomposed into uncountably many imaginary sorts each of
which does admit such a metric.
Therefore, at the price of possibly working with a multi sorted
language, we may assume that all sorts admit a definable metric.
For convenience we will proceed as if there is a single home sort, but
the generalisation to many sorts should be obvious.

Let us fix, once and for all, a definable metric on the home sort.
By \cite{pezz:Morley} we know that any two such metrics are uniformly
equivalent, so notions such as supersimplicity and superstability are
not affected by our choice of metric.
If the reader wishes nevertheless to avoid such an arbitrary choice,
she or he may use the notion of abstract distances from
\cite{pezz:Morley} instead of real-valued distances whose
interpretation depends on a metric function.

Distances on tuples will be viewed as tuples of distances of
singletons:
\begin{dfn}
  \label{dfn:DistTuple}
  Let $I$ be a set of indices.
  \begin{enumerate}
  \item 
    If $\bar a$ and $\bar b$ are $I$-tuples then we consider
    $d(\bar a,\bar b)$ to be the $I$-tuple
    $(d(a_i,b_i)\colon i \in I) \in [0,\infty]^I$
    (In fact, a definable metric is necessarily bounded, so we can
    replace $\infty$ with some real number here, but keeping $\infty$ as a
    special distance is convenient).
  \item If $\bar \varepsilon, \bar \varepsilon' \in [0,\infty]^I$, we say that
    $\bar \varepsilon \leq \bar \varepsilon'$ if $\varepsilon_i \leq \varepsilon'_i$ for all $i \in I$.
  \item If $\bar \varepsilon, \bar \varepsilon' \in [0,\infty]^I$, we say that
    $\bar \varepsilon < \bar \varepsilon'$ if $\varepsilon_i < \varepsilon'_i$ for all $i \in I$,
    \emph{and} $\varepsilon_i' = \infty$ for all but finitely many $i \in I$.
    \\
    For the purpose of this definition we use the convention
    that $\infty < \infty$.
  \item Given
    $\bar \varepsilon \in [0,\infty]^I$ and $\varepsilon' \in [0,\infty]$, we understand
    statements such as $\bar \varepsilon < \varepsilon'$, $\bar \varepsilon \leq \varepsilon'$, etc., by
    replacing $\varepsilon'$ with the $I$-tuple all of whose coordinates are
    $\varepsilon'$.
  \end{enumerate}
\end{dfn}
From our convention that $\infty < \infty$ it follows that $\bar \varepsilon < \infty$ for all
$\bar \varepsilon \in [0,\infty]^I$, and $\bar \varepsilon > 0 \Longrightarrow \bar \varepsilon > \frac{1}{2}\bar \varepsilon$
(where $\frac{1}{2}(\varepsilon_i)_{i\in I} = (\frac{\varepsilon_i}{2})_{i\in I}$, and
$\frac{\infty}{2} = \infty < \infty$).

Superstability and supersimplicity in the first order context deal
with properties of independence of finite tuples of elements.
In the metric setting we replace ``finite tuple'' with a ``virtually
finite'' one:

\begin{dfn}
  \label{dfn:Virtual}
  \begin{enumerate}
  \item 
    A \emph{virtual element} is formally a pair $(a,\varepsilon)$, where $a$ is a
    singleton and $\varepsilon \in [0,\infty]$.
    Usually a virtual element $(a,\varepsilon)$ will be denoted by $a^\varepsilon$, and we
    think of this conceptually as ``the element $a$ up to distance
    $\varepsilon$''.
  \item 
    A \emph{virtual tuple} is a tuple of virtual elements,
    i.e., an object of the form $(a_i^{\varepsilon_i}\colon i \in I)$.
    This can also be denoted by $\bar a^{\bar \varepsilon}$, or simply $a^\varepsilon$,
    as single lowercase letters may denote arbitrary tuples.
  \item As in \fref{dfn:DistTuple}, if $\bar a$ is an $I$-tuple, and
    $\varepsilon \in [0,\infty]$ a single distance, we understand $\bar a^\varepsilon$ as
    $\bar a^{\bar \varepsilon}$, where $\bar \varepsilon$ is a tuple consisting of $I$
    repetitions of $\varepsilon$.
  \item A \emph{virtually finite tuple} is a virtual tuple
    $\bar a^{\bar \varepsilon}$ such that $\bar \varepsilon > 0$.
    We remind the reader that according to
    \fref{dfn:DistTuple},
    this means that $\varepsilon_i > 0$ for all $i$,
    and $\varepsilon_i = \infty$ for all but finitely many $i \in I$.
  \end{enumerate}
\end{dfn}

\begin{ntn}
  Unless explicitly said otherwise, $a$, $b$, etc., denote possibly
  infinite tuples of elements in a model.
  Similarly, $\varepsilon$, $\delta$, etc., denote possibly infinite tuples in
  $[0,\infty]$.
  Thus $a^\varepsilon$ denotes a virtual tuple, possibly infinite, with the
  implicit understanding that $a$ and $\varepsilon$ are of the same length.

  If we wish to render explicit the fact that these are tuples we may
  use notation such as $\bar a^{\bar \varepsilon}$ etc.
\end{ntn}

We identify a tuple $a$ with the virtual tuple $a^0$: knowing $a$ up
to distance $0$ means knowing $a$ precisely.
More generally, if the relation $d(x,y) \leq \varepsilon$ is transitive (e.g., in
the rare case where  the metric is an ultrametric) then it is an
equivalence relation, and we can identify the virtual tuple $a^\varepsilon$
with the hyperimaginary $a/[d(x,y) \leq \varepsilon]$.

Similarly, we identify a virtual tuple $\bar a^{\bar \varepsilon}$ with any
virtual tuple obtained by omitting or adding virtual elements of the
form $a_i^\infty$: knowing $a_i$ up to distance $\infty$ means not knowing
$a_i$ at all.
(The reader will see that these identifications are consistent with
the way we use virtual tuples later on.)

Note that every virtually finite $I$-tuple $\bar a^{\bar \varepsilon}$ can be
thus identified with the sub-tuple corresponding to
$J = \{i\in I\colon \varepsilon_i < \infty\}$, which is finite as $\bar \varepsilon > 0$.

This identification allows a convenient re-definition of the notion of
a sub-tuple:
\begin{dfn}
  A \emph{virtual sub-tuple} of a virtual tuple $a^\varepsilon$ is a virtual
  tuple (which can be identified with) $a^{\varepsilon'}$ for some $\varepsilon' \geq \varepsilon$.
\end{dfn}

\begin{rmk}
  The notion of a virtual sub-tuple extends the ``ordinary'' notion of
  sub-tuple.
  Indeed, let $\bar a^{\bar \varepsilon}$ be a virtual $I$-tuple, and
  $\bar b^{\bar \delta}$ a sub-tuple in the ordinary sense, i.e., given by
  restricting so a subset of indices $J \subseteq I$.
  For $i \in I$ define $\varepsilon'_i = \varepsilon_i$ if
  $i \in J$, and $\varepsilon'_i = \infty$ otherwise.
  Then $\bar \varepsilon' \geq \bar \varepsilon$, and $\bar b^{\bar \delta}$ can be identified
  with the virtual sub-tuple $\bar a^{\bar \varepsilon'}$.
\end{rmk}

We define types of virtual tuples:
\begin{dfn}
  As we defined in \cite{pezz:Morley}, if $p(x)$ is a partial type in
  a tuple of variables $x$, and $\varepsilon$ is a tuple of distances of the
  same length, then $p(x^\varepsilon)$ is defined as the partial type
  $\exists y\, (p(y) \land d(x,y) \leq \varepsilon)$.
  Since $d(x,y) \leq \varepsilon$ is a type-definable property, this is indeed
  expressible by a partial type.

  We define $\tp(a^\varepsilon)$ as $p(x^\varepsilon)$ where $p = \tp(a)$.
  Similarly, if $p(x,y) = \tp(a,b)$ then $\tp(a^\varepsilon/b^\delta)$ is
  $p(x^\varepsilon,b^\delta)$.
\end{dfn}

\begin{rmk}
  If $a' \models \tp(a^\varepsilon/b^\delta)$ then we say that $a^\varepsilon$ and ${a'}^\varepsilon$ have the
  same type over $b^\delta$, in symbols $a^\varepsilon \equiv_{b^\delta} {a'}^\varepsilon$.
  This is a symmetric relation.
\end{rmk}
\begin{proof}
  Assume that $a' \models \tp(a^\varepsilon/b^\delta)$.
  Then there are $a''b' \equiv ab$ such that $d(a'b,a''b') \leq \varepsilon\delta$.
  Let $a''',b''$ be such that $a''b'a'b \equiv aba'''b''$.
  Then $d(ab,a'''b'') \leq \varepsilon\delta$ and $a'''b'' \equiv a'b$, whereby
  $a \models \tp({a'}^\varepsilon/b^\delta)$.
\end{proof}


We recall from \cite{pezz:Morley}:
\begin{dfn}
  $T$ is supersimple if for every virtually finite tuple $a^\varepsilon$
  and set $A$, there is a finite subset $A_0 \subseteq A$ such that
  $\tp(a^\varepsilon/A)$ does not divide over $A_0$.
\end{dfn}
If we replace ``virtually finite tuple'' with ``virtually finite
singleton'' (i.e., $a$ is a singleton) we obtain an equivalent
definition.

We now turn to the principal new definition in this paper, of
independence \emph{over} virtual tuples.

\begin{dfn}
  We say that an indiscernible sequence $(b_i : i < \omega)$ \emph{could
    be in $\tp(b/c^\rho)$} if there are $(c_i)$ such that:
  \begin{enumerate}
  \item $(b_ic_i : i < \omega)$ is an indiscernible sequence in $\tp(bc)$.
  \item For all $i,j \leq \omega$: $d(c_i,c_j) \leq \rho$.
  \end{enumerate}
\end{dfn}

\begin{rmk}
  \label{rmk:ERDist}
  Let $E_\rho(x,y)$ be the relation $d(x,y) \leq \rho$.
  Assume that $E_\rho$ happens to be transitive, and therefore an
  equivalence relation (this would happen in the rare case that the
  metric is an ultrametric, and also if $\rho = 0$).
  Then a sequence $(b_i : i < \omega)$ could be in $\tp(b/c^\rho)$
  if and only if it has an automorphic image in $\tp(b/(c/E_\rho))$.

  In particular, if $\rho = 0$ then $E$ is equality, and $(b_i)$ could
  be in $\tp(b/c^0)$ if and only if it has an automorphic image in
  $\tp(b/c)$.
  This justifies the terminology, as well as the identification
  between $c^0$ and $c$.
\end{rmk}

\begin{dfn}
  We say that $a^\varepsilon \ind_{c^\rho} b$ if every indiscernible sequence
  that could be in $\tp(b/c^\rho)$ could also be in $\tp(b/a^\varepsilon c^\rho)$.
\end{dfn}
As explained in the introduction, this notion of independence has very
few ``nice'' properties, although these suffice for the application we
seek.
We will not dare to extend it to, say, independence of the form
$a^\varepsilon \ind_{c^\rho} b^\delta$ without being able to show that such extension
has useful properties.

When restricting to independence over non-virtual (real or even
hyperimaginary) tuples, it is not true that $a^\varepsilon \ind_b c$  if and
only of $\tp(a^\varepsilon/bc)$ does not divide over $c$.
These notions are close enough to being equivalent, though:
\begin{lem}
  \label{lem:eInd}
  Assume that $T$ is simple.
  For $a^\varepsilon$, $b$ and $c$, the following conditions imply one another
  from top to bottom:
  \begin{enumerate}
  \item $a^\varepsilon \ind_c b$ (remember to identify $c$ with $c^0$).
  \item There is a Morley sequence for $b$ over $c$ which could be in
    $\tp(b/a^\varepsilon c)$.
  \item $\tp(a^\varepsilon/bc)$ does not divide over $c$.
  \item $a^{2\varepsilon} \ind_c b$.
  \end{enumerate}
\end{lem}
\begin{proof}
  \begin{cycprf}
  \item[\impnext] By definition.
  \item[\impnext] We recall the $\Dx{-}$ ranks from
    \cite{pezz:fnctcat}:
    Fixing the tuple $x$, we define $\Xi = \Xi(x)$ as the set of all pairs
    $(\varphi(x,y),\psi(y_{<k}))$ ($y$ and $k$ may vary) such that
    $\varphi$ and $\psi$ are positive formulae and $\psi$ is a $k$-inconsistency
    witness for $\varphi$, i.e.,
    \begin{gather*}
      T \vdash \lnot\exists xy_{<k} \,\left( \psi(y_{<k}) \land \bigwedge_{i<k} \varphi(x,y_i) \right).
    \end{gather*}
    If $p(x)$ is a partial type with parameters in $A$,
    $\Dx{p}$ is a subset of $\Xi^{<|T|^+}$ such that
    for $\xi_{<\alpha} \in \Xi^\alpha$:
    \begin{itemize}
    \item $\alpha = 0$: $\xi_{<\alpha} \in \Dx{p}$ if and only if $p$ is consistent.
    \item $\alpha$ limit: $\xi_{<\alpha} \in \Dx{p}$ if and only if
      $\xi_{<\beta} \in \Dx{p}$ for all $\beta <\alpha$.
    \item $\alpha = \beta+1$, $\xi_\beta = (\varphi(x,y),\psi(y_{<k}))$:
      $\xi_{<\alpha} \in \Dx{p}$ if and only if
      there exists an $A$-indiscernible sequence $(b_i\colon i < \omega)$ in the
      sort of $y$ such that $\models \psi(b_{<k})$ and $\xi_{<\beta} \in \Dx{p\cup\{\varphi(x,b_0)\}}$.
    \end{itemize}
    We recall that this rank characterises independence (for $T$
    simple and thick, and thus in particular simple and Hausdorff):
    If $A \subseteq B$ then $p \in \tS(B)$
    does not divide over $A$ if and only if
    $\Dx{p} = \Dx{p\rest_A}$.

    So let $(b_i\colon i < \omega)$ be a Morley sequence for $b$ over $c$ which
    could be in $\tp(b/a^\varepsilon c)$.
    Then there are $(a_i\colon i < \omega)$ and $c'$
    such that $(a_ib_ic'\colon i < \omega)$ is an indiscernible
    sequence in $\tp(abc)$ and $d(a_i,a_j) \leq \varepsilon$ for all $i,j<\omega$.
    Extend the sequence to length $\omega+1$.
    Then by standard arguments we have $b_\omega \ind_{b_{<\omega}} a_{<\omega}c'$, so:
    \begin{gather*}
      \Dx{b_\omega/b_{<\omega}} = \Dx{b_\omega/b_{<\omega}a_{<\omega}c'} \subseteq \Dx{b_\omega/a_0c'} \subseteq
      \Dx{b_\omega/c'}.
    \end{gather*}
    On the other hand, since $(b_i)$ is a Morley sequence over $c$:
    \begin{gather*}
      \Dx{b_\omega/b_{<\omega}} = \Dx{b/c} = \Dx{b_\omega/c'}
    \end{gather*}
    Therefore equality holds all the way and we have $b_\omega \ind_{c'}
    a_0$.
    Since $d(a_0,a_\omega) \leq \varepsilon$, it follows that
    $\tp(a_\omega^\varepsilon/b_\omega c')$ does not divide over $c'$, and by invariance
    $\tp(a^\varepsilon/bc)$ does not divide over $c$.
  \item[\impnext] We assume that $\tp(a^\varepsilon/bc)$ does not divide over
    $c$.
    Then there exists $a'$ such that $a' \ind_c b$ and $d(a,a') \leq
    \varepsilon$.
    Let $(b_i)$ be any indiscernible sequence that could be in
    $\tp(b/c)$.
    Then we might as well assume that it is in $\tp(b/c)$ and since
    $a' \ind_c b$ it can even be in $\tp(b/a'c)$.
    Find now $(a_i)$ such that $a_ib_i \equiv_{a'c} ab$.
    Then we may always choose them such that $(a_ib_i)$ is
    $c$-indiscernible, and $d(a_i,a') \leq \varepsilon$ for all $i$ yields
    $d(a_i,a_j) \leq 2\varepsilon$ for all $i,j$, as required.
  \end{cycprf}
\end{proof}

This means that $a \ind_c b$ if and only if $\tp(a/bc)$ does not
divide over $c$ (since $2\cdot0 = 0$), so this definition agrees with the
usual definition of independence of ordinary (i.e., non-virtual)
elements.

We can continue \fref{rmk:ERDist} to show that if the tuple
of distances $\rho$ defines an equivalence relation $E_\rho$
then $a^\varepsilon \ind_{c^\rho} b$ if and only if $a^\varepsilon \ind_{c/E_\rho} b$ (here
$c/E_\rho$ is viewed as hyperimaginary, rather than virtual).
If $\varepsilon$ also defines an equivalence relation $E_\varepsilon$, then
$a^\varepsilon \ind_{c^\rho} b$ if and only if $a/E_\varepsilon \ind_{c/E_\rho} b$.

Also, \fref{lem:eInd} and the fact that
$\varepsilon > 0 \Longrightarrow \frac{1}{2}\varepsilon > 0$ give:

\begin{prp}
  \label{prp:TSS}
  $T$ is supersimple if and only if for every virtually finite
  tuple (singleton) $a^\varepsilon$
  and set $A$ there is a finite subset $A_0 \subseteq A$ such that
  $a^\varepsilon \ind_{A_0} A$.
\end{prp}

\begin{prp}
  \label{prp:eIndTrans}
  Independence satisfies right downward transitivity, left upward
  transitivity, and two-sided monotonicity:
  \begin{enumerate}
  \item If $a^\varepsilon \ind_{c^\rho} b$ and $\delta$ is any tuple of distances of
    the length of $b$, then $a^\varepsilon \ind_{b^\delta c^\rho} b$.
  \item If $a^\varepsilon \ind_{c^\rho} b$ and $d^\upsilon \ind_{a^\varepsilon c^\rho} b$ then
    $a^\varepsilon d^\upsilon \ind_{c^\rho} b$.
  \item If $a^\varepsilon \ind_{c^\rho} bd$ and $\varepsilon' \geq \varepsilon$ (i.e., if $a^{\varepsilon'}$
    is a virtual sub-tuple of $a^\varepsilon$) then
    $a^{\varepsilon'} \ind_{c^\rho} b$.
  \end{enumerate}
\end{prp}
\begin{proof}
  \begin{enumerate}
  \item Let $(b_i)$ be an indiscernible sequence that could be in
    $\tp(b/b^\delta c^\rho)$.
    This is the same as saying that $(b_i)$ could be in
    $\tp(b/c^\rho)$ and $d(b_i,b_j) \leq \delta$ for all $i,j<\omega$.
    As we assume that $a^\varepsilon \ind_{c^\rho} b$, the sequence $(b_i)$ could
    be in $\tp(b/a^\varepsilon c^\rho)$; since
    $d(b_i,b_j) \leq \delta$ it could also be in $\tp(b/a^\varepsilon b^\delta c^\rho)$, as
    required.
  \item If $(b_i)$ is indiscernible and could be in $\tp(b/c^\rho)$ then
    it could also be in $\tp(b/a^\varepsilon c^\rho)$ and therefore in
    $\tp(b/a^\varepsilon c^\rho d^\upsilon)$.
  \item Let $(b_i)$ be an indiscernible sequence that could be in
    $\tp(b/c^\rho)$.
    By standard arguments we can find $(d_i)$ such that $(b_id_i)$ is
    indiscernible and could be in $\tp(bd/c^\rho)$.
    As we assume that $a^\varepsilon \ind_{c^\rho} bd$, it could also be in
    $\tp(bd/a^\varepsilon c^\rho)$.
    Therefore $(b_i)$ could be in $\tp(b/a^\varepsilon c^\rho)$ and \emph{a
      fortiori} in $\tp(b/a^{\varepsilon'}c^\rho)$.
  \end{enumerate}
\end{proof}

We obtain a more general form, in this context, of the finite
character of independence:
\begin{prp}
  \label{prp:FinChar}
  For all $a^\varepsilon$, $b$ and $c^\rho$: $a^\varepsilon \ind_{c^\rho} b$ if and only if
  $a^{\varepsilon'} \ind_{c^\rho} b'$ for all $\varepsilon' > \varepsilon$ and finite $b' \subseteq b$.
  (By our approach, $a^{\varepsilon'}$ should be viewed as a finite
  sub-tuple of $a^\varepsilon$, since $\varepsilon' > \varepsilon$.)

  In particular, $a \ind_c b$ if and only if
  $a^\varepsilon \ind_c b$ for all $\varepsilon > 0$.
\end{prp}
\begin{proof}
  Left to right is by monotonicity.
  For right to left, assume that $a^\varepsilon \nind_{c^\rho} b$.
  Then there is an indiscernible sequence
  $(b_i)$ that could be in $\tp(b/c^\rho)$ but not in $\tp(b/a^\varepsilon c^\rho)$.
  Letting $p(x,y,z) = \tp(a,b,c)$, the latter means that the following
  is inconsistent:
  \begin{gather*}
    \bigwedge_{i<\omega} p(x_i,b_i,z_i) \land
    \bigwedge_{i,j<\omega} [d(x_i,x_j) \leq \varepsilon \land d(z_i,z_j) \leq \rho].
  \end{gather*}
  Since $d(x_i,x_j) \leq \varepsilon$ is logically equivalent to
  $\bigwedge_{\varepsilon'>\varepsilon} d(x_i,x_j) \leq \varepsilon'$, and the family of all $\varepsilon' > \varepsilon$ is
  closed for finite infima, we obtain by compactness some
  $\varepsilon' > \varepsilon$ such that the above is still inconsistent with $\varepsilon'$
  instead of $\varepsilon$.
  Therefore $a^{\varepsilon'} \nind_{c^\rho} b$.
  Replacing $b$ with a finite sub-tuple is similar (and fairly
  standard).
\end{proof}

Let us recall from   \cite[Lemma~1.2]{pezz:catsim}
the following useful fact about ``extraction'' of
indiscernible sequence from long sequences:
\begin{fct}
  \label{fct:ERind}
  Let $A$ be a set of parameters, and $\lambda \geq \beth_{|\tS_\kappa(A)|^+}$.
  Then
  for any sequence $(a_i : i < \lambda)$ of $\kappa$-tuples there is an
  $A$-indiscernible sequence $(b_i : i < \omega)$ such that for all
  $n < \omega$ there are $i_0 < \ldots < i_{n-1} < \lambda$ for which
  $\tp(b_0\ldots b_{n-1}/A) = \tp(a_{i_0}\ldots a_{i_{n-1}}/A)$.
\end{fct}

\begin{thm}
  \label{thm:SSChar}
  $T$ is supersimple if and only if for every virtually finite tuple
  (singleton) $a^\varepsilon$, and any tuple $b$,
  there is a virtually finite sub-tuple $b^\delta$ of $b$ (i.e., there
  exists $\delta > 0$ of the appropriate length) such that
  $a^\varepsilon \ind_{b^\delta} b$.
\end{thm}
\begin{proof}
  In order to prove right to left, it would suffice to show that for
  every virtually finite singleton $a^\varepsilon$ and set $B$ there is
  $B_0 \subseteq B$ finite such that $a^\varepsilon \ind_{B_0} B$.

  Let $b = \bar b$ be an $I$-tuple enumerating $B$.
  Then by assumption there is $\delta = \bar \delta > 0$ such that
  $a^\varepsilon \ind_{b^\delta} b$.
  Let $B_0 = \{b_i\colon \delta_i < \infty\}$.
  Then $B_0$ is finite, and by right downward transitivity:
  $a^\varepsilon \ind_{B_0} B$.

  We now prove left to right: aiming for a contradiction, we assume
  that $T$ is supersimple, and yet there is no $\delta$ as in the
  statement.
  Let $p(x,y) = \tp(a,b)$, $q(y) = \tp(b)$.
  We will construct by induction a sequence of tuples
  $(b_n\colon n < \omega)$ in
  $q$, and a sequence of tuples of distances $(\delta_n\colon n < \omega)$.
  These will satisfy, among other things, that $\delta_n \geq 2\delta_{n+1} > 0$ and
  $d(b_n,b_{n+1}) \leq \delta_n$.

  For convenience, let $\delta_{-1} = \infty$.

  At the $n$th step, assume we already have $b_{<n}$ satisfying
  $q$ and $\delta_{n-1} > 0$.
  By assumption $a^\varepsilon \nind_{b^{\delta_{n-1}}} b$.
  Therefore there is an indiscernible sequence $(b_n^i: i < \omega)$
  such that $d(b_n^i,b_n^j) \leq \delta_{n-1}$ for all $i,j < \omega$ and yet the
  following is inconsistent:
  \begin{gather*}
    \tag{*$_n$} \label{eq:dind1}
    \bigwedge_{i < \omega} p(x^i,b_n^i) \land \bigwedge_{i,j<\omega} d(x^i,x^j) \leq \varepsilon
  \end{gather*}
  By a compactness argument as in the proof of
  \fref{prp:FinChar}, there exists $\delta_n > 0$ such that the
  following weakening of \fref{eq:dind1} is still inconsistent:
  \begin{gather*}
    \tag{**$_n$} \label{eq:dind2}
    \bigwedge_{i < \omega} p(x^i,{b_n^i}^{2\delta_n}) \land
    \bigwedge_{i,j<\omega} d(x^i,x^j) \leq \varepsilon
  \end{gather*}
  We may always assume that $\delta_n \leq \frac{1}{2}\delta_{n-1}$.

  If $n = 0$, the sequence $(b_n^i\colon i < \omega)$ is $b_{<n}$-indiscernible, and we
  skip the following paragraph.
  If $n > 0$, note that all that matters for \fref{eq:dind2} is
  the type of the sequence $(b_n^i\colon i < \omega)$: we may therefore
  replace it with another sequence which has the same type, such that
  in addition $(b_n^i\colon i < \omega)$ is
  $b_{<n}$-indiscernible and satisfy $d(b_n^i,b_{n-1}) \leq \delta_{n-1}$.
  In order to see this, extend this sequence to arbitrary length
  $\lambda+1$ $(b_n^i\colon i \leq \lambda)$.
  Since $b_n^\lambda \equiv b \equiv b_{n-1}$, we may assume (up to replacing
  $(b_n^i\colon i \leq \lambda)$ with an automorphic image) that $b_n^\lambda = b_{n-1}$.
  Applying \fref{fct:ERind} to the sequence
  $(b_n^i\colon i < \lambda)$ over $b_{<n}$, we can find a sequence
  $(c_n^i\colon i < \omega)$ which is $b_{<n}$-indiscernible, and such
  that for all $m < \omega$ there are $i_0<\cdots<i_{m-1} < \lambda$ such that
  $$c_n^0,\ldots,c_n^{m-1} \equiv_{b_{<n}} b_n^{i_0},\ldots,b_n^{i_{m-1}}.$$
  Therefore $(c_n^i\colon i < \omega) \equiv (b_n^i\colon i < \omega)$, and
  $d(c_n^i,b_{n-1}) \leq \delta_{n-1}$, so
  the sequence $(c_n^i\colon i < \omega)$ has the required properties.

  Let $b_n = b_n^0$, so in particular
  $d(b_n,b_{n-1}) \leq \delta_{n-1}$, and continue the construction.

  Once the construction is done, let $\delta_\omega = \inf \delta_n$.
  If $b$ is an $I$-tuple then so are $\delta_n = \delta_{n,\in I}$ and
  $\delta_\omega = \delta_{\omega,\in I}$.
  Since $\delta_n \geq 2\delta_{n+1}$ for all $n$, we must have
  $\delta_{\omega,i} \in \{0,\infty\}$ for all $i \in I$.
  But if $i \in I$ is such that $\delta_{\omega,i} = \infty$, then $\delta_{n,i} = \infty$
  for all $n$, which means that the $i$th coordinate of $b$ and the
  $b_n$ played absolutely no role throughout the construction, and may
  be entirely dropped.
  Therefore, replacing $b$, $b_n$, $\delta_n$, etc., with sub-tuples we
  may assume that $\delta_\omega = 0$.

  The fact that $\delta_n \geq 2\delta_{n+1}$ implies that for every $n \leq m$:
  $d(b_n,b_m) \leq 2\delta_n$.
  Thus the partial type $q(y) \land \bigwedge_n d(b_n,y) \leq 2\delta_n$ is
  consistent, and has a realisation $b_\omega$.
  Since $\inf \delta_n = 0$, $b_\omega$ is the unique realisation of
  this type, so $b_\omega \in \dcl(b_{<\omega})$
  (we say that $b_\omega$ is the limit of the
  Cauchy sequence $(b_n\colon n < \omega)$).
  Since $b_\omega \equiv b$, there is $a_\omega$ such that $\models p(a_\omega,b_\omega)$.
  By supersimplicity there is $n$ such that
  $a_\omega^\varepsilon \ind_{b_{<n}} b_{<\omega}$, whereby $a_\omega^\varepsilon \ind_{b_{<n}} b_\omega$.

  Let us go back to our $b_{<n}$-indiscernible sequence
  $(b_n^i\colon i < \omega)$, and we recall that $b_n = b_n^0$.
  Find $(b_\omega^i\colon i < \omega)$ such that $b_\omega^0 = b_\omega$ and
  $(b_n^ib_\omega^i\colon i < \omega)$ is
  $b_{<n}$-indiscernible.
  Since $a_\omega^\varepsilon \ind_{b_{<n}} b_\omega$, there are $(a_\omega^i\colon i < \omega)$
  realising
  \begin{gather*}
    \bigwedge_{i<\omega} p(a_\omega^i,b_\omega^i) \land \bigwedge_{i,j<\omega} d(a_\omega^i,a_\omega^j) \leq \varepsilon
  \end{gather*}
  But $d(b_n,b_\omega) \leq 2\delta_n \Longrightarrow d(b_n^i,b_\omega^i) \leq 2\delta_n$.
  This shows that:
  \begin{gather*}
    \bigwedge_{i < \omega} p(a_\omega^i,{b_n^i}^{2\delta_n})
    \land \bigwedge_{i,j<\omega} d(a_\omega^i,a_\omega^j) \leq \varepsilon,
  \end{gather*}
  so \fref{eq:dind2} was consistent after all.
\end{proof}

Having given meaning to $a^\varepsilon \ind_c b$, it is natural to define the
corresponding $\SU$-ranks:

\begin{dfn}
  $\SU(a^\varepsilon/b)$ is the minimal rank taking ordinal values or $\infty$
  satisfying:
  \begin{quote}
    $\SU(a^\varepsilon/b) \geq \alpha+1$ if and only if there is $c$ such that $a^\varepsilon
    \nind_b c$ and $\SU(a^\varepsilon/bc) \geq \alpha$.
  \end{quote}
\end{dfn}

\begin{prp}
  \begin{enumerate}
  \item $T$ is supersimple if and only if $\SU(a^\varepsilon/b) < \infty$ for all
    $b$ and virtually finite $a^\varepsilon$.
  \item Assuming that $T$ is supersimple, $a \ind_b c \Longleftrightarrow \SU(a^\varepsilon/bc)
    = \SU(a^\varepsilon/b)$ for all $\varepsilon>0$.
  \end{enumerate}
\end{prp}
\begin{proof}
  \begin{enumerate}
  \item Standard argument, using \fref{prp:TSS}.
  \item If $\SU(a^\varepsilon/bc) = \SU(a^\varepsilon/b)$ for all $\varepsilon>0$, then
    $a^\varepsilon \ind_b c$ for all $\varepsilon>0$, whereby $a \ind_b c$ by the finite
    character (\fref{prp:FinChar}).

    Conversely, assume that $a \ind_b c$.
    Clearly, $\SU(a^\varepsilon/b) \geq \SU(a^\varepsilon/bc)$.
    We prove by induction on $\alpha$ that $\SU(a^\varepsilon/b) \geq \alpha \Longrightarrow
    \SU(a^\varepsilon/bc) \geq \alpha$.
    For $\alpha = 0$ and limit this is clear, so we need to prove for
    $\alpha = \beta+1$.

    Since $\SU(a^\varepsilon/b) \geq \beta+1$, there is $d$ such that
    $\SU(a^\varepsilon/bd) \geq \beta$ and $a^\varepsilon \nind_b d$.
    We may assume that $d \ind_{ab} c$.
    We assumed that $a \ind_b c$, whereby $ad \ind_b c$ and
    $a \ind_{bd} c$.
    Therefore, by the induction hypothesis, $\SU(a^\varepsilon/bcd) \geq \beta$.
    On the other hand, $a^\varepsilon \nind_{bc} d$: otherwise we'd get
    $a^\varepsilon c \ind_b d$ by \fref{prp:eIndTrans}, contradicting prior
    assumptions.
    This shows that $\SU(a^\varepsilon/bc) \geq \beta+1 = \alpha$.
  \end{enumerate}
\end{proof}

\begin{qst}
  It is fairly easy to prove that for all virtually finite $a^\varepsilon$ and
  $b^\delta$, and all tuples $c$:
  \begin{gather*}
    \SU(a^\varepsilon/bc) + \SU(b^\delta/c) \leq \SU(a^\varepsilon b^\delta/c).
  \end{gather*}
  Note, however, that we use $\SU(a^\varepsilon/bc)$ rather than
  $\SU(a^\varepsilon/b^\delta c)$, to which we haven't given a meaning.
  This is a serious problem, since the converse inequality may easily
  be false (for example, if $a = b$ and $\delta = \infty$.):
  \begin{gather*}
    \SU(a^\varepsilon b^\delta/c) \nleq \SU(a^\varepsilon/bc) \oplus \SU(b^\delta/c).
  \end{gather*}
  Is there a way to give meaning to $\SU(a^\varepsilon/b^\delta c)$ such that
  the standard Lascar inequalities (or reasonable variants thereof)
  hold?
\end{qst}

\section{Lovely pairs}
\label{sec:SSPairs}

We assume familiarity at least with the basics of lovely pairs as
exposed in \cite{ppv:pairs}, where for every simple first order theory
$T$ we constructed its theory of lovely pairs $T^P$, and proved that
if $T$ has the weak non-finite-cover-property then $T^P$ is simple and
independence in $T^P$ was characterised.
In \cite{pezz:nfopairs} we generalised the latter result to the case
where $T$ is any thick simple cat.
Namely, for each such $T$ we constructed a cat $T^\fP$ whose
$|T|^+$-saturated models are precisely the lovely pairs of
models of $T$, and proved it is simple with the same characterisation
of independence.
If $T$ is a first order theory then
$T^\fP$ is first order if and only if $T$ has the
weak non-finite-cover-property, in which case $T^\fP$ coincides with
$T^P$.

\begin{conv}
  If $(M,P)$ is a lovely pair of models of $T$ and $a \in M$ then
  $\tp(a)$ denotes the type of $a$ in $M$ (in the sense of $T$) while
  $\tp^\fP(a)$ denotes its type in $(M,P)$ (i.e., in the sense of
  $T^\fP$).
\end{conv}

We are going to use a few results from \cite{pezz:nfopairs} which do
not appear explicitly in \cite{ppv:pairs}.

If $(M,P)$ is a lovely pair and $a \in M$, then $\tp^\fP(a)$
determines the set of all possible types of Morley sequences (both in
the sense of $T$) for $a$ over $P(M)$.
Conversely, any of these types determines $\tp^\fP(a)$.
Also, the property ``the sequence $(a_i\colon i<\omega)$ has the type of a
Morley sequence for $a$ over $P$'' is definable by a partial type in
$x_{<\omega}$, which is denoted by $\mcl(a)$ (the \emph{Morley class} of
$a$).

Finally, let $a,b,c \in M$, and $(a_ib_ic_i\colon i < \omega) \models \mcl(abc)$ be
a sequence in some model of $T$.
Then $a$ is independent from $b$ over $c$ in the sense of $T^\fP$, in
symbols $a \ind[\fP]_c b$, if and only if
$a_{<\omega} \ind_{c_{<\omega}} b_{<\omega}$ (here in the sense of $T$).

\begin{conv}
  $T$ is a simple Hausdorff cat, and in particular thick.
  Therefore $T^\fP$ exists and the properties mentioned above hold.
\end{conv}

Since $T$ is Hausdorff, so is $T^\fP$.
Also, as $T$ is a reduct of $T^\fP$, any definable metric we might
have fixed for $T$ is also a definable metric in the sense of $T^\fP$.
(Since $T^\fP$ is richer, there may be new definable metrics: however,
as all definable metrics are uniformly equivalent, this makes no
difference.)

\begin{lem}
  \label{lem:TPeDiv}
  Let $a$, $b$ and $c$ be tuples in a lovely pair $(M,P)$, and let
  $(a_ib_ic_i\colon i < \omega) \models \mcl(abc)$ in a universal domain for $T$.
  Let $a^\varepsilon$ be a virtual sub-tuple of $a$.
  Then $\tp^\fP(a^\varepsilon/bc)$ divides over $c$ (in the sense of $T^\fP$) if and
  only if $\tp(a_{<\omega}^\varepsilon/b_{<\omega}c_{<\omega})$ divides over $c_{<\omega}$ (in the
  sense of $T$).

  As we said earlier, the notation $a_{<\omega}^\varepsilon$ here means the virtual
  tuple $(a_i^\varepsilon\colon i < \omega)$, and similarly for tuples of other
  lengths.
  As $\varepsilon$ is a tuple of distances of the length of $a$, there
  should be no ambiguity about this.
\end{lem}
\begin{proof}
  Assume that $\tp^\fP(a^\varepsilon/bc)$ does not divide over $c$.
  Then there is $a' \in M$ such that $d(a,a') \leq \varepsilon$ and
  $a' \ind[\fP]_c b$.
  There exist $(a_i'\colon i < \omega)$ such that
  $(a_i'a_ib_ic_i\colon i < \omega) \models \mcl(a'abc)$.
  Then $a' \ind[\fP]_c b \Longrightarrow a'_{<\omega} \ind_{c_{<\omega}} b_{<\omega}$, and
  $d(a,a') \leq \varepsilon \Longrightarrow d(a_i,a_i') \leq \varepsilon$ for all $i < \omega$, whereby
  $\tp(a_{<\omega}^\varepsilon/b_{<\omega}c_{<\omega})$ does not divide over $c_{<\omega}$.

  Conversely, assume that $\tp(a_{<\omega}^\varepsilon/b_{<\omega}c_{<\omega})$ does not
  divide over $c_{<\omega}$.
  Then there exist $(a_i'\colon i < \omega)$ such that $d(a_i,a_i') < \varepsilon$ for
  all $i < \omega$ and $a'_{<\omega} \ind_{c_{<\omega}} b_{<\omega}$, but the sequence
  $(a'_ia_ib_ic_i\colon i < \omega)$ needs not be indiscernible.

  Continue the sequence $(a_ib_ic_i)$ to an indiscernible sequence of
  length $\lambda > \omega$, big enough to allow us to apply \fref{fct:ERind}
  later on.
  Let $e = \Cb(b_\omega c_\omega/b_{<\omega}c_{<\omega})$
  and $f = \Cb(c_\omega/c_{<\omega})$ (in the sense of $T$), so
  $(b_ic_i\colon i < \lambda)$ is a Morley sequence over $e$,
  and $(c_i\colon i < \lambda)$ is a Morley sequence over $f$, and indiscernible
  over $ef$, whereby $c_{<\lambda} \ind_f e$.
  It follows that for all $w \subseteq \omega$:
  \begin{align*}
    b_{\in w}c_{\in w} \ind_e c_{\in\omega\setminus w}
    & \Longrightarrow b_{\in w} \ind_{c_{\in w}e} c_{<\omega}
    \Longrightarrow b_{\in w}e \ind_{c_{\in w}f} c_{<\omega}
  \end{align*}
  Whereby:
  \begin{align*}
    b_{<\omega} \ind_{c_{<\omega}} a'_{<\omega} & \Longrightarrow b_{\in w} \ind_{c_{<\omega}} a'_{<\omega}
    \Longrightarrow b_{\in w} \ind_{fc_{\in w}} a'_{\in w}.
  \end{align*}

  For any finite $w \subseteq \lambda$, $\tp(b_{\in w}c_{\in w}f)$ depends solely on
  $|w|$ (where the tuples $b_{\in w}$ and $c_{\in w}$ are enumerated
  according to the ordering induced on $w$ from $\lambda$).
  Thus by the
  definability of independence for known complete types, for every
  $n < \omega$ and tuple of variables $x$ there is a partial type
  $\rho_{n,x}(x,y_{<n},z_{<n},f)$ such that for every $w \in [\lambda]^n$ and
  every $g$ in the sort of $x$:
  \begin{gather*}
    b_{\in w} \ind_{c_{\in w}f} g \Longleftrightarrow \, \models \rho_{n,x}(g,b_{\in w},c_{\in w},f).
  \end{gather*}

  Putting these two facts together and applying compactness we can
  find a sequence $(a''_i\colon i < \lambda)$
  such that $d(a_i,a''_i) \leq \varepsilon$ for all $i < \lambda$, and
  $b_{\in w} \ind_{c_{\in w}f} a''_{\in w}$ for every finite $w \subseteq \lambda$.

  As we could have chosen $\lambda$ arbitrarily big, by standard extraction
  arguments (i.e., \fref{fct:ERind}) there exists a
  sequence $(\tilde a_i\colon i < \omega)$ such that
  $(\tilde a_ia_ib_ic_i\colon i < \omega)$   is $f$-indiscernible, and in addition
  $d(a_i,\tilde a_i) \leq \varepsilon$ for all
  $i < \omega$, and $b_{\in w} \ind_{c_{\in w}f} \tilde a_{\in w}$ for all
  finite $w \subseteq \omega$.
  As every formula in $\tp(\tilde a_{<\omega}/b_{<\omega}c_{<\omega}f)$ only involves
  finitely many variables and parameters, it does not divide over
  $c_{\in w}f$ for some finite $w \subseteq \omega$, and \emph{a fortiori} over
  $c_{<\omega}f$.
  Therefore $\tilde a_{<\omega} \ind_{c_{<\omega}f} b_{<\omega}$.
  As $f \in \dcl(c_{<\omega})$, we conclude that
  $\tilde a_{<\omega} \ind_{c_{<\omega}} b_{<\omega}$.

  As the sequence $(\tilde a_ia_ib_ic_i\colon i < \omega)$ is indiscernible, we
  may extend it to length $\omega+1$.
  Then one can find a lovely pair $(M,P)$ such that
  $\tilde a_{\leq\omega}a_{\leq\omega}b_{\leq\omega}c_{\leq\omega} \in M$,
  $\tilde a_{<\omega}a_{<\omega}b_{<\omega}c_{<\omega} \in P$, and
  $\tilde a_\omega a_\omega b_\omega c_\omega \ind_{\tilde a_{<\omega}a_{<\omega}b_{<\omega}c_{<\omega}} P$.
  It follows that
  $(\tilde a_ia_ib_ic_i\colon i < \omega) \models \mcl^{(M,P)}(\tilde a_\omega a_\omega b_\omega c_\omega)$.
  Thus $d(a_\omega,\tilde a_\omega) \leq \varepsilon$ and $\tilde a_\omega \ind[\fP]_{c_\omega} b_\omega$.
  In particular, $\tp^\fP({a_\omega}^\varepsilon/b_\omega c_\omega)$ does not divide over $c_\omega$.

  As we assumed that $(a_ib_ic_i\colon i < \omega) \models \mcl(abc)$, we have
  $a_\omega b_\omega c_\omega \equiv^\fP abc$, so $\tp^\fP(a^\varepsilon/bc)$ does not divide over $c$.
\end{proof}

\begin{thm}
  \label{thm:TPSS}
  If $T$ is supersimple, then so is $T^\fP$.
\end{thm}
\begin{proof}
  We need to show that for every virtually finite element $a^\varepsilon$
  and every tuple $B = b_{\in I}$ in a model of $T^\fP$,
  there exists a finite sub-tuple $B' \subseteq B$ such that
  $\tp^\fP(a^\varepsilon/B)$ does not divide over $B'$.

  We follow the path of \cite[Corollary~3.6]{pezz:nfopairs}.
  Choose $(a_jB_j\colon j<2\omega) \models \mcl(aB)$ in some model of $T$.
  Let $b_{\in I,j}$ be the enumeration of each $B_j$ corresponding to
  $B = b_{\in I}$.

  By supersimplicity of $T$ there are tuples of distances
  $\upsilon = \upsilon_{<\omega} > 0$ and $\rho = \rho_{\in I,<2\omega} > 0$ such that:
  \begin{gather}
    \label{eq:TPSS1}
    a_\omega^\varepsilon \ind_{(a_{<\omega})^\upsilon(B_{<2\omega})^\rho} a_{<\omega}B_{<2\omega}.
  \end{gather}
  Then by definition, there are only finitely many $j < \omega$ such that
  $\upsilon_j \neq \infty$, and only finitely many pairs $(i,j) \in I \times 2\omega$ such
  that $\rho_{i,j} \neq \infty$, so we can define:
  \begin{align*}
    n & = 1+\max\{j<\omega\colon \upsilon_j \neq \infty
    \text{ or there exists $i \in I$ such that } \rho_{i,j} \neq \infty\},\\
    \delta & = \min\{\varepsilon, \upsilon_j\colon j < n\},\\
    J_0 & = \{i \in I\colon \rho_{i,j} \neq \infty \text{ for some } j < 2\omega\}.
  \end{align*}
  In particular, $\varepsilon \geq \delta > 0$ and $J_0 \subseteq I$ is finite.

  By right downward transitivity, \fref{eq:TPSS1} becomes:
  \begin{gather*}
    a_\omega^\varepsilon \ind_{a_{<n}^\delta, b_{\in J_0,\in[0,n)\cup[\omega,2\omega)}} a_{<\omega}B_{<2\omega}.
  \end{gather*}
  Applying supersimplicity again, there is $J_1 \subseteq I$ finite
  such that:
  \begin{gather*}
    a_{<n}^\delta \ind_{b_{\in J_1,<\omega}} B_{<\omega}.
  \end{gather*}
  Let $J = J_0 \cup J_1 \subseteq I$, and let $B' = b_{\in J}$, $B'_j = b_{\in J,j}$.
  These are finite sub-tuples of $B$ and $B_j$, respectively, and:
  \begin{gather}
    \label{eq:TPSS2}
    a_\omega^\varepsilon \ind_{a_{<n}^\delta, B'_{\in[0,n)\cup[\omega,2\omega)}} a_{<\omega}B_{<2\omega},\\
    \label{eq:TPSS3}
    a_{<n}^\delta \ind_{B'_{<\omega}} B_{<\omega}.
  \end{gather}

  We now prove by induction on $n \leq m < \omega$ that:
  \begin{gather}
    \label{eq:TPSS4}
    a_{<n}^\delta a_{<m}^\varepsilon \ind_{B'_{<\omega}} B_{<\omega}.
  \end{gather}
  For $m = n$, this follows from \fref{eq:TPSS3} since
  $\varepsilon \geq \delta$.
  Assume now \fref{eq:TPSS3} for some $n \leq m < \omega$.
  From \fref{eq:TPSS2} we obtain by monotonicity and right
  downward transitivity:
  \begin{gather}
    \label{eq:TPSS5}
    a_\omega^\varepsilon \ind_{a_{<n}^\delta a_{<m}^\varepsilon, B'_{\in[0,m)\cup[\omega,2\omega)}} B_{\in[0,m)\cup[\omega,2\omega)}.
  \end{gather}
  Since $(a_j,b_{\in I,j}\colon j<2\omega)$ is an indiscernible sequence, there is
  an automorphism sending $a_{\omega+j}B_{\omega+j}$ to $a_{m+j}B_{m+j}$
  for every $j < \omega$, while keeping $a_{<m}B_{<m}$ in place.
  Applying such an automorphism to \fref{eq:TPSS5} we get:
  \begin{gather*}
    a_m^\varepsilon \ind_{a_{<n}^\delta a_{<m}^\varepsilon, B'_{<\omega}} B_{<\omega}.
  \end{gather*}
  By left upward transitivity and the induction assumption
  \fref{eq:TPSS4} we obtain:
  \begin{gather*}
    a_{<n}^\delta a_{<m+1}^\varepsilon \ind_{B'_{<\omega}} B_{<\omega}.
  \end{gather*}

  This concludes the proof of \fref{eq:TPSS4}.
  By the finite character of independence
  (\fref{prp:FinChar}) we conclude that:
  \begin{gather*}
    a_{<n}^\delta a_{<\omega}^\varepsilon \ind_{B'_{<\omega}} B_{<\omega}.
  \end{gather*}
  In particular, $\tp(a_{<\omega}^\varepsilon/B_{<\omega})$ does not divide over
  $B'_{<\omega}$.
  By \fref{lem:TPeDiv} $\tp^\fP(a^\varepsilon/B)$ does not divide over $B'$,
  which is finite, as required.
\end{proof}

\begin{cor}
  If $T$ is superstable, then so is $T^\fP$.
\end{cor}
\begin{proof}
  $T$ is superstable if and only if it is stable and supersimple.
  In that case $T^\fP$ is stable by \cite[Theorem~3.10]{pezz:nfopairs}
  and supersimple by \fref{thm:TPSS}.
  Therefore $T^\fP$ is superstable.
\end{proof}

\begin{qst}
  Assume that $T$ is $\omega$-stable.
  Is $T^\fP$ $\omega$-stable as well?
\end{qst}

\providecommand{\bysame}{\leavevmode\hbox to3em{\hrulefill}\thinspace}
\providecommand{\MR}{\relax\ifhmode\unskip\space\fi MR }
\providecommand{\MRhref}[2]{%
  \href{http://www.ams.org/mathscinet-getitem?mr=#1}{#2}
}
\providecommand{\href}[2]{#2}

\end{document}